\documentclass[a4paper,11pt]{article}

\usepackage[left=2.5cm,right=2.5cm,top=3cm,bottom=3cm,pdftex]{geometry}
\usepackage{amssymb}
\usepackage{amsmath}
\usepackage{amsthm}
\usepackage{dsfont}
\usepackage{url}
\usepackage{natbib}

\usepackage[utf8]{inputenc}
\usepackage[T1]{fontenc}

\usepackage[pdftex]{graphicx}
\DeclareGraphicsExtensions{.png, .pdf, .jpg} % {.jpg,.pdf,.mps,.png,.gif}

\usepackage[pdftex, colorlinks, linkcolor=blue, urlcolor=blue, citecolor=blue, breaklinks=true]{hyperref}

\newtheorem{thm}{Theorem}
\newtheorem{lem}{Lemma}
\newtheorem{fol}{Corollary}

\begin{document}

\begin{center}
\Large\textbf{Various Optimality Criteria for the Prediction of Individual Response Curves
} \\[11pt]
\normalsize
Maryna Prus\footnote{Maryna Prus: \href{mailto:maryna.prus@ovgu.de}{maryna.prus@ovgu.de}}\\[11pt]

\footnotesize
Otto-von-Guericke University Magdeburg, Institute for Mathematical Stochastics, 
\\
PF 4120, D-39016 Magdeburg, Germany\\
\normalsize
\end{center}

\begin{quote}
\textbf{Abstract:} We consider optimal designs for the Kiefer's cirteria (and, in particular, the E-criterion) and the G-criterion in random coefficients regression (RCR) models. We obtain general the Kiefer's criteria for approximate designs and prove the equivalence of the E-criteria in the fixed effects and RCR models. We discuss in detail the G-criterion for ordinary linear regression on specific design regions.

\end{quote}

\section{Introduction}

The subject of this work is random coefficients regression (RCR) models, which were initially introduced in biosciences and are nowdays popular in many fields of statistical application, for example in agricultural studies or medical research. 

Linear criteria as well as a generalized version of the \textit{D}-criterion in RCR have been discussed in \cite{pru1}. Some results for the particular case of the c-criterion, interpolation and extrapolation, are presented in \cite{pru4}. 

In this work we concentrate on the Kiefer's $\Phi_q$-criteria (see e.\,g. \cite{fed2}, p.~54, for the fixed effects case), which are based on the eigenvalues of the information matrix in the random coefficient regression models, and the \textit{G}- (global) criterion, which aims to minimize the maximal prediction mean squared error over the experimental region.

For fixed effects models, the \textit{E}-criterion (see e.\,g \cite{atk1}, ch.~10), which can be recognized as the particular Kiefer's $\Phi_q$-criterion for $q\rightarrow\infty$, has been considered in detail in \cite{kie}. \cite{kie2} have proved the equivalence of \textit{D}- and \textit{G}-criteria in models with homoscedastic errors. More recently, \cite{wen} has discussed the \textit{G}-optimality in heteroscedastic models. For that models, it has been established that the \textit{D}- and \textit{G}-criteria are in general not equivalent.

Here, we present some results for the general form of Kiefer’s $\Phi_q$-criteria in RCR models and prove for the \textit{E}-criterion for largest eigenvalue that the optimal designs for fixed effects models retain their optimality for the prediction. We consider the \textit{G}-criterion for ordinary linear regression with a diagonal covariance structure on specific experimental regions.

The paper has the following structure: In the second part the random coefficients regression models will be specified and the best linear unbiased prediction of individual random parameters will be presented. The third part provides analytical results for the designs, which are optimal for the prediction. The paper will be concluded by a short discussion in the forth part.

\section{Model Specification and Prediction}\label{k2}

In this paper we consider the random coefficients regression models, in which the $j$-th observation of individual $i$ is given by
\begin{equation}\label{mod}
	{Y}_{ij}=\mathbf{f}(x_{j})^\top \mbox{\boldmath{$\beta $}}_i+ \varepsilon_{ij},\quad j=1, \dots, m, \quad i=1, \dots, n, \quad x\in \mathcal{X}
\end{equation}
where $m$ is the number of observations per individual, $n$ is the number of individuals,\linebreak $\mathbf{f} =(f_1, \dots,f_p)^\top$ is a vector of known regression functions and $\mathcal{X}$ is an experimental region.
The observational errors $\varepsilon_{ij}$ are assumed to have zero mean and common variance $\sigma^2>0$. 
The individual parameters $\mbox{\boldmath{$\beta $}}_i=( \beta_{i1}, \dots, \beta_{ip})^\top$ have unknown population mean $\mathrm{E}\,(\mbox{\boldmath{$\beta $}}_i)={\mbox{\boldmath{$\beta $}}}$ and known positive definite covariance matrix $\mathrm{Cov}\,(\mbox{\boldmath{$\beta $}}_i)=\sigma^2\mathbf{D}$. 
All individual parameters $\mbox{\boldmath{$\beta $}}_{i}$ and all observational errors $\varepsilon_{ij}$ are assumed to be uncorrelated.

The best linear unbiased predictor of individual parameter $\mbox{\boldmath{$\beta $}}_i$ is given by
\begin{equation}\label{blup}
\hat{\mbox{\boldmath{$\beta $}}}_i=(\mathbf{F}^\top \mathbf{F}+\mathbf{D}^{-1})^{-1}(\mathbf{F}^\top \mathbf{F}\,\hat{\mbox{\boldmath{$\beta $}}}_{i;{\rm ind}}+\mathbf{D}^{-1}\hat{\mbox{\boldmath{$\beta $}}}),
\end{equation}
where $\hat{\mbox{\boldmath{$\beta $}}}_{i;{\mathrm{ind}}}=(\mathbf{F}^\top \mathbf{F})^{-1}\mathbf{F}^\top \mathbf{Y}_i$ and $\hat{\mbox{\boldmath{$\beta $}}}=(\mathbf{F}^\top \mathbf{F})^{-1}\mathbf{F}^\top \bar{\mathbf{Y}} $ for the individual vector of observations $\mathbf{Y}_i=(Y_{i1}, \dots,Y_{im})^\top$, the mean observational vector $\bar{\mathbf{Y}} =\frac{1}{n}\sum_{i=1}^n{\mathbf{Y}_i}$ and the design matrix $\mathbf{F}=(\mathbf{f}(x_1), \dots, \mathbf{f}(x_m))^\top$, which is assumed to be of full column rank.

The mean squared error matrix of the of the vector $\hat{\mathbf{B}}=(\hat{\mbox{\boldmath{$\beta $}}}_1^\top, \dots,\hat{\mbox{\boldmath{$\beta $}}}_n^\top)^\top$ of all predictors of all individual parameters is given by
\begin{equation}\label{mse}
\mathrm{MSE}= \sigma^2\left(\frac{1}{n}\left(\mathds{1}_n\mathds{1}_n^\top \right)\otimes  \left(\mathbf{F}^\top \mathbf{F}\right)^{-1} + \left(\mathds{I}_n-{\textstyle{\frac{1}{n}}}\mathds{1}_{n}\mathds{1}_{n}^\top \right)\otimes \left(\mathbf{F}^\top \mathbf{F}+\mathbf{D}^{-1}\right)^{-1} \right),
\end{equation}
where $\mathds{I}_n$ is identity matrix, $\mathds{1}_n$ is the vector of length $n$ with all elements equal to $1$ and $\otimes$ denotes the Kronecker product.

\section{Optimal Designs}\label{k3}

We define here exact designs in the following way:
\begin{equation}\label{design}
\xi= \left( \begin{array}{ccc}  x_1 & , \dots, & x_k \\  m_1 &, \dots,& m_k \end{array} \right),
\end{equation} 
where $x_1, \dots, x_k$ are the distinct experimental settings with the numbers of replications $m_1, \dots, m_k$. Approximate designs are defined as
\begin{equation}\label{designa}
\xi= \left( \begin{array}{ccc}  x_1 & , \dots, & x_k \\  w_1 &, \dots,& w_k \end{array} \right),
\end{equation}
where $w_j=m_j/m$ and only the conditions $w_j\geq 0$ and $\sum_{j=1}^{k}w_j=1$ have to be satisfied (integer numbers of replications are not required). Further we will use the notation
\begin{equation}\label{m_xi}
\mathbf{M}(\xi)=\frac{1}{m}\sum_{j=1}^k m_j\mathbf{f}(x_j)\mathbf{f}(x_j)^\top
%=\frac{1}{m}\mathbf{F}^\top \mathbf{F} \, ,
\end{equation} 
for the standardized information matrix from the fixed effects model and $\mbox{\boldmath{$\Delta $}}=m\, \mathbf{D}$ for the adjusted dispersion matrix for the random effects. We assume the matrix $\mathbf{M}(\xi)$ to be non-singular. With this notation the definition of the mean squared error \eqref{mse} can be extended for approximate designs  to
\begin{equation}\label{msea}
\mathrm{MSE}(\xi)= {\frac{1}{n}}\left(\mathds{1}_n\mathds{1}_n^\top \right)\otimes  \mathbf{M}(\xi)^{-1} + \left(\mathds{I}_n-{\frac{1}{n}}\mathds{1}_{n}\mathds{1}_{n}^\top \right)\otimes \left(\mathbf{M}(\xi)+\mathbf{\Delta}^{-1}\right)^{-1}, 
\end{equation} 
when we neglect the constant term $\frac{\sigma^2}{m}$. 

\subsection{Kiefer's criterion}

We define the Kiefer's $\Phi_q$-criterion for the prediction in RCR models for all values of $q>0$ as the following function of the
%eigenvalues $\lambda_1, \dots, \lambda_{np}$ 
trace of the MSE matrix \eqref{mse} of the prediction:
\begin{equation}\label{ki}
\Phi_q=\left(\frac{1}{np}\,\mathrm{tr}\,\left(\mathrm{MSE}^{q}\right)\right)^{\frac{1}{q}}, \quad q \in (0,\infty).
\end{equation}
%\begin{equation}\label{ki}
%\Phi_q=\left(\sum_{k=1}^{np}{\lambda_k^q}\right)^{\frac{1}{q}}, \quad q \in (0,\infty)
%\end{equation}
Then we extend the definition of the criterion for approximate designs and obtain the following result.
\begin{thm}\label{t1}
The Kiefer's $\Phi_q$-criterion for the prediction of individual parameters is for approximate designs given by
\begin{equation}\label{kia}
\Phi_q(\xi)=\left(\frac{1}{np}\left(\mathrm{tr}\,\left(\mathbf{M}(\xi)^{-q}\right)+(n-1)\,\mathrm{tr}\,\left(\left(\mathbf{M}(\xi)+\mathbf{\Delta}^{-1}\right)^{-q}\right)\right)\right)^{\frac{1}{q}}, \quad q \in (0,\infty).
\end{equation}
\end{thm}
\begin{proof}
As proved in \cite{pru3}, ch.~5, the $np$ eigenvalues of $\mathrm{MSE}(\xi)$ are the $p$ eigenvalues $\eta_1, \dots, \eta_p$ of $\mathbf{M}(\xi)^{-1}$ with multiplicity $1$ and the $p$ eigenvalues $\mu_1, \dots, \mu_p$ of $\left(\mathbf{M}(\xi)+\mathbf{\Delta}^{-1}\right)^{-1}$ with multiplicity $n-1$. Then we obtain
\begin{eqnarray*}
\Phi_q(\xi)&=&\left(\frac{1}{np}\,\mathrm{tr}\,\left(\mathrm{MSE}(\xi)^{q}\right)\right)^{\frac{1}{q}}\\
&=&\left(\frac{1}{np}\left(\sum_{s=1}^{p}{\eta_s^q}+(n-1)\sum_{s=1}^{p}{\mu_s^q}\right)\right)^{\frac{1}{q}}\\
&=& \left(\frac{1}{np}\left(\mathrm{tr}\,\left(\mathbf{M}(\xi)^{-q}\right)+(n-1)\,\mathrm{tr}\,\left(\left(\mathbf{M}(\xi)+\mathbf{\Delta}^{-1}\right)^{-q}\right)\right)\right)^{\frac{1}{q}}.
\end{eqnarray*}
\end{proof}
Note that the Kiefer's criterion \eqref{kia} can be recognized as a weighted sum of the Kiefer's $\Phi_q$-criteria in fixed effects and Bayesian models (if we neglect the power $\frac{1}{q}$) . The weight of the Bayesian part increases with increasing number of individuals $n$. For models with only one individual optimal designs in fixed effects models are optimal for the prediction.

Particular cases of the Kiefer's $\Phi_q$-criterion (for $q\rightarrow 0$ and $q=1$), \textit{D}- and \textit{A}-criteria, have been considered in detail by \cite{pru1}. The \textit{E}-criterion, which can also be recognized as the limiting  Kiefer's criterion  for $q\rightarrow \infty$, will be discussed in the next section.

\subsection{\textit{E}-criterion}
We define the \textit{E}-criterion (eigenvalue criterion) for the prediction as the largest eigenvalue $\lambda_{max}$ of the MSE matrix \eqref{mse}:
\begin{equation}\label{e}
\Phi_E=\lambda_{max}.
\end{equation}
For approximate designs we obtain the following result.
\begin{thm}
The \textit{E}-criterion for the prediction of individual parameters is for approximate designs given by
\begin{equation}\label{ea}
\Phi_E(\xi)=\eta_{max}(\xi),
\end{equation}
where $\eta_{max}(\xi)$ denotes the largest eigenvalue of $\mathbf{M}(\xi)^{-1}$.
\end{thm}
\begin{proof}
It is easy to see that $\mathbf{M}(\xi)<\mathbf{M}(\xi)+\mathbf{\Delta}^{-1}$ in Loewner ordering. Therefore, the smallest eigenvalue $\zeta_{min}(\xi)$ of $\mathbf{M}(\xi)$ is smaller than the smallest eigenvalue $\nu_{min}(\xi)$ of $\mathbf{M}(\xi)+\mathbf{\Delta}^{-1}$ (see e.\,g. \cite{fed2}, p.~11) and, consequently,  the largest eigenvalue $\eta_{max}(\xi)=\frac{1}{\zeta_{min}(\xi)}$ of $\mathbf{M}(\xi)^{-1}$ is larger than the largest eigenvalue $\mu_{max}(\xi)=\frac{1}{\nu_{min}(\xi)}$ of $\left(\mathbf{M}(\xi)+\mathbf{\Delta}^{-1}\right)^{-1}$. Then making use of the proof of Theorem~\ref{t1} we obtain the result \eqref{ea}.
\end{proof}
\begin{fol}
\textit{E}-optimal designs in the fixed effects model are \textit{E}-optimal for the prediction of individual parameters in the random coefficient regression model.
\end{fol}

\subsection{\textit{G}-criterion}
For the prediction in RCR models we define the \textit{G}-criterion (global criterion) as the maximal sum of individual expected squared differences of the predicted and real response across all individuals with respect to all possible observational settings:
\begin{equation}\label{g}
\Phi_{G} =\max_{x\in \mathcal{X}}\sum_{i=1}^{n}\mathrm{E}\left( \left(\mathbf{f}(x)^\top (\hat{\mbox{\boldmath{$\beta $}}}_i-\mbox{\boldmath{$\beta $}}_i)\right)^2 \right).
\end{equation}
We receive the following criterion for the approximate designs.
\begin{thm}
The \textit{G}-criterion for the prediction of individual parameters is for approximate designs given by
\begin{equation}\label{ga}
\Phi_{G}(\xi) = \max_{x\in \mathcal{X}}\left(\mathbf{f}(x)^\top\left(\mathbf{M}(\xi)^{-1}+(n-1)\left(\mathbf{M}(\xi)+\mathbf{\Delta}^{-1}\right)^{-1}\right)\mathbf{f}(x)\right).
\end{equation}
\end{thm}
\begin{proof}
\begin{eqnarray*}
\Phi_{G} &=&\max_{x\in \mathcal{X}}\sum_{i=1}^{n}\mathrm{var}\left(\mathbf{f}(x)^\top (\hat{\mbox{\boldmath{$\beta $}}}_i-\mbox{\boldmath{$\beta $}}_i) \right)\\
&=&\max_{x\in \mathcal{X}}\sum_{i=1}^{n}\mathrm{tr}\left(\mathrm{Cov} (\hat{\mbox{\boldmath{$\beta $}}}_i-\mbox{\boldmath{$\beta $}}_i)\, \mathbf{f}(x)\,\mathbf{f}(x)^\top\right)\\
&=&\max_{x\in \mathcal{X}}\mathrm{tr}\left(\mathrm{MSE} \left(\mathbb{I}_n\otimes \left(\mathbf{f}(x)\,\mathbf{f}(x)^\top\right)\right)\right)\\
\end{eqnarray*}
\begin{eqnarray*}
\Phi_{G}(\xi) &=&\max_{x\in \mathcal{X}}\mathrm{tr}\left(\mathrm{MSE}(\xi) \left(\mathbb{I}_n\otimes (\mathbf{f}(x)\,\mathbf{f}(x)^\top)\right)\right)\\
&=&\max_{x\in \mathcal{X}}\left(\mathrm{tr}\left(\mathbf{M}(\xi)^{-1}\mathbf{f}(x)\,\mathbf{f}(x)^\top\right)+(n-1)\,\mathrm{tr}\left(\left(\mathbf{M}(\xi)+\mathbf{\Delta}^{-1}\right)^{-1}\mathbf{f}(x)\,\mathbf{f}(x)^\top\right)\right)\\
&=& \max_{x\in \mathcal{X}}\left(\mathbf{f}(x)^\top\left(\mathbf{M}(\xi)^{-1}+(n-1)\left(\mathbf{M}(\xi)+\mathbf{\Delta}^{-1}\right)^{-1}\right)\mathbf{f}(x)\right)
\end{eqnarray*}
\end{proof}
%For further considerations we rewrite this criterion in terms of the information matrix $\mathbf{M}$:
%\begin{equation}\label{gm}
%\Phi_{G}(\mathbf{M}) = \max_{x\in \mathcal{X}}\left(\mathbf{f}(x)^\top\left(\mathbf{M}^{-1}+(n-1)\left(\mathbf{M}+\mathbf{\Delta}^{-1}\right)^{-1}\right)\mathbf{f}(x)\right).
%\end{equation}
The following property can be easily verified for the $G$-criterion \eqref{ga} .
\begin{lem}\label{l1}
Let $\xi_1$ and $\xi_2$ be approximate designs of form \eqref{designa} so that $\mathbf{M}(\xi_1)\leq\mathbf{M}(\xi_2)$ in Loewner ordering. Then it holds for the \textit{G}-criterion \eqref{ga} that $\Phi_{G}(\xi_1)\geq\Phi_{G}(\xi_2)$.
\end{lem}
Note that the \textit{G}-criterion \eqref{ga} is not differentiable and, therefore, no optimality condition in the sense of an equivalence theorem (see \cite{kie2}) is straightforward to formulate. Therefore, we will consider in detail the following particular model .
\vspace{2mm}

\textbf{Particular model: straight line regression}

We consider the linear regression model
\begin{equation}\label{lr}
	Y_{ij}= \beta_{i1}+\beta_{i2}x_j+\varepsilon_{ij} 
\end{equation}
for two different experimental regions $\mathcal{X}_1=[0,a]$, $a>0$, and $\mathcal{X}_2=[-b,b]$, $b>0$. 

For this model the function
\begin{equation}
\Phi(x, \xi) = \mathbf{f}(x)^\top\left(\mathbf{M}(\xi)^{-1}+(n-1)\left(\mathbf{M}(\xi)+\mathbf{\Delta}^{-1}\right)^{-1}\right)\mathbf{f}(x),
\end{equation}
which can be also recognized as the sensitivity function of the \textit{D}-criterion in RCR models (see \cite{pru1}), is a parabola with a positive leading term (with respect to $x$). Therefore, $\Phi(x, \xi)$ achieves its maxima at the ends of the intervals. Then the \textit{G}-criterion \eqref{ga} simplifies to
\begin{equation*}
\Phi_G(\xi_1) = \mathrm{max}\left\{\Phi(0, \xi_1),\Phi(a, \xi_1)\right\}
\end{equation*}
or
\begin{equation*}
\Phi_G(\xi_2) = \mathrm{max}\left\{\Phi(-b, \xi_2),\Phi(b, \xi_2)\right\}
\end{equation*}
for all non-singular designs $\xi_1$ on $\mathcal{X}_1$ or $\xi_2$ on $\mathcal{X}_2$, respectively. 

Further we will make use of the following simple lemmas.
\begin{lem}\label{l2}
Let
\begin{equation*}
\xi= \left( \begin{array}{ccc}  x_1 & , \dots, & x_k \\  w_1 &, \dots,& w_k \end{array} \right)
\end{equation*}
be an approximate design in model \eqref{lr} on $\mathcal{X}_1$. Then it holds for the approximate design
\begin{equation*}
\xi'= \left( \begin{array}{ccc}  0 & a \\  1-w' & w' \end{array} \right),
\end{equation*}
where $w'=\frac{1}{a}\sum_{j=1}^k{x_jw_j}$, that $\mathbf{M}(\xi)\leq\mathbf{M}(\xi')$ in Loewner ordering.
\end{lem}
\begin{lem}\label{l3}
Let
\begin{equation*}
\xi= \left( \begin{array}{ccc}  x_1 & , \dots, & x_k \\  w_1 &, \dots,& w_k \end{array} \right)
\end{equation*}
be an approximate design in model \eqref{lr} on $\mathcal{X}_2$. Then it holds for the approximate design
\begin{equation*}
\xi'= \left( \begin{array}{ccc}  -b & b \\  1-w' & w' \end{array} \right),
\end{equation*}
where $w'=\frac{1}{2}\left(\frac{1}{b}\sum_{j=1}^k{x_jw_j}+1\right)$, that $\mathbf{M}(\xi)\leq\mathbf{M}(\xi')$ in Loewner ordering.
\end{lem}
Then it follows directly from Lemmas~\ref{l1}, \ref{l2} and \ref{l3} that at least one of $G$-optimal designs in model \eqref{lr} on $\mathcal{X}_1$ or $\mathcal{X}_2$ is of the form 
\begin{equation*}
\xi_1= \left( \begin{array}{ccc}  0 & a \\  1-w_1 & w_1 \end{array} \right)
\end{equation*}
or
\begin{equation*}
\xi_2= \left( \begin{array}{ccc}  -b & b \\  1-w_2 & w_2 \end{array} \right),
\end{equation*}
respectively. Then only the optimal weights $w_1^*$ and $w_2^*$ have to be determined.

Further we will additionally assume the diagonal structure of the covariance matrix of the random effects: $\mathbf{D}=\mathrm{diag}(d_1,d_2)$. Then it is easy to verify that $\Phi(0, \xi_1)$ and $\Phi(-b, \xi_2)$ increase and $\Phi(a, \xi_1)$ and $\Phi(b, \xi_2)$ decrease with increasing values of $w_1$ and $w_2$, respectively. Consequently, the $G$-optimal designs are solutions of the equations
\begin{equation}\label{oc1}
\Phi(0, \xi_1)=\Phi(a, \xi_1).
\end{equation}
and
\begin{equation}\label{oc2}
\Phi(-b, \xi_2)=\Phi(b, \xi_2).
\end{equation}
Note that if equation \eqref{oc1} or \eqref{oc2} has no solutions, the resulting optimal designs on $\mathcal{X}_1$ or $\mathcal{X}_2$, respectively, will lead to a singular information matrix.

Then equation \eqref{oc1} may be represented for $c_s=1/(md_s)$, $s=1,2$, in form
\begin{equation}\label{oc11}
\frac{2w_1-1}{w_1(1-w_1)}=(n-1)\frac{a^2(1+c_1-2w_1)}{(1+c_1)(a^2w_1+c_2)-a^2w_1^2}.
\end{equation}
For condition \eqref{oc2} we obtain 
\begin{equation}\label{oc21}
\frac{2w_2-1}{w_2(1-w_2)}=(n-1)\frac{4b^2(1-2w_2)}{(1+c_1)(b^2+c_2)-b^2(2w_2-1)^2}.
\end{equation}
It is easy to see that the only solution of \eqref{oc21} is given by the optimal weight $w_2^*=0.5$.

According to  \cite{pru1}, the optimality condition for the $D$-criterion is given by 
\begin{equation}\label{oc}
\Phi(x, \xi)\leq p+(n-1)\left(\left(\mathbf{M}(\xi)+\mathbf{\Delta}^{-1}\right)^{-1}\mathbf{M}(\xi)\right),
\end{equation}
for all $x\in\mathcal{X}$, and equality in \eqref{oc} for all support points. This condition coincides with \eqref{oc11} for $\mathcal{X}=\mathcal{X}_1$ and \eqref{oc21} for $\mathcal{X}=\mathcal{X}_2$, which implies for both design regions the equivalence of the \textit{D}- and \textit{G}-criteria in the linear regression model \eqref{lr} with the diagonal covariance structure.

\section{Discussion}

We have discussed Kiefer's the $\Phi_q$-criteria and the global (\textit{G}-) criterion in RCR models. The obtained general form of the Kiefer's criterion can be recognized as the weighted sum of the Kiefer's criterion in fixed effects and Bayesian models, where the weight of the Baysian part increases with increasing number of individuals. For the \textit{E}-criterion (particular Kiefer's $\Phi_q$-criterion with $q\rightarrow\infty$), it was proved that optimal designs in fixed effects models retain their optimality for the prediction in RCR models. The \textit{G}-criterion cannot be factorized in the fixed effects and Bayesian parts. This criterion has been discussed in detail for ordinary linear regression on specific experimental regions, $[0,a]$, $a>0$, and $[-b,b]$, $b>0$. For this special case, the equivalence of \textit{D}- and \textit{G}-criteria has been established. However, the equivalence of these two criteria does not hold in general for RCR models.

\section*{Acknowledgments} This research has been supported by grant SCHW 531/16-1 of the German Research Foundation (DFG). The author thanks to  Radoslav Harman and Rainer Schwabe for fruitful discussions.

\bibliographystyle{natbib}
\bibliography{prus2}

\end{document}